\documentclass[11pt]{amsart}
\usepackage{amsmath,amsfonts,amsthm,amscd,amssymb,graphicx}

\newtheorem{theorem}{Theorem}[section]

\newtheorem{lemma}[theorem]{Lemma}

\newtheorem{corollary}[theorem]{Corollary}

\begin{document}

\title[Uniform convergence and the free central limit theorem.]
{Uniform convergence and the free central limit theorem}

\author[John Williams]{John Williams}

\address{Department of Mathematics, Indiana University, Bloomington, IN 47405}
\email{jw32@indiana.edu}

\maketitle

\begin{abstract}
We prove results about uniform convergence of densities in the free
central limit theorem without assumptions of boundedness on the
support.
\end{abstract}


\newcommand{\pn}{{$\sqrt{n}\varphi_{\mu}(\sqrt{n}z)$}}
\newcommand{\pp}{{$\varphi_{\mu_{n}}(z)$}}
\newcommand{\CC}{\mathbb{C}}
\newcommand{\RR}{\mathbb{R}}

\section{Introduction}
Consider free, identically distributed random variables $X_{1},
X_{2}, \ldots $ such that $E(X_{j}) = 0$ and $E(X_{j}^{2}) = 1$. The
free central limit theorem tells us that the distribution $\mu_{n}$
of the random variable $n^{-1/2} (X_{1} + X_{2} + \cdots + X_{n})$
converges weakly to the semicircle law $\gamma$, given by
$$d\gamma (t) = (2 \pi)^{-1} \sqrt {4 - t^{2}} \chi _{(-2,2)} (t) dt.$$
It was shown in \cite{BV2} that, when the variables $X_{j}$ are
bounded, $\mu_{n}$ is absolutely continuous for big enough $n$, and
its density converges to $(2 \pi)^{-1} \sqrt{4 - x^{2}}$ uniformly.

It is our purpose to extend this result to unbounded random
variables.  In particular, we show that $d\mu_{n} / dt$ converges to
$(2 \pi)^{-1} \sqrt{4 - x^{2}}$ uniformly on compact subsets of
$(-2,2)$.  When $\mu$ is infinitely divisible, we show that this
convergence is uniform on $\RR$.  The main results have recently
been superseded in \cite{Wg} with a proof making stong use of free
brownian motion.

\section{Preliminaries.}
A noncommutative probability space is defined to be a pair
$(A,\varphi)$, where $A$ is a $W^{*}$-algebra and $\varphi$ is a
tracial state on $A$. A bounded random variable is an element $a\in
A$. A family of unital subalgebras $\{A_{i}\}_{i \in I} \subseteq A$
is said to be freely independent if $\varphi(a_{1}a_{2} \cdots
a_{n}) = 0$ whenever $a_{j} \in A_{i_{j}}$ with $i_{j} \neq i_{j+1}
$ for $j = 1 , 2, \ldots , n-1$ and $\varphi(a_{k}) = 0$ for $k =
1,2, \ldots ,n$.

Given a random variable $x\in A$, its distribution is the functional
$\mu_{x} : \CC[X] \rightarrow \CC$ defined by the property that
$\mu_{x} (P(X)) = \varphi (P(a))$  for $P \in \CC[X]$.  If $x =
x^{*}$, then $\mu_{x}$ is given by integration against a compactly
supported probability measure on the real line. More specifically,
let $E_{x}$ denote the spectral measure of $x$, and define a Borel
measure $\mu$ on $\RR$ by setting $\mu (\sigma) = \varphi (E_{x}
(\sigma))$.  Then we have $\varphi(P(x)) = \int_{\sigma (x)} p(t) d
\mu (t)$ where $\sigma(x)$ denotes the spectrum of our random
variable

Now, consider self-adjoint, freely independent random variables $x,y
\in A$ and their corresponding probability distributions, $\mu_{x}$
and $\mu_{y}$.  It was established in \cite{V2} that $\mu_{x+y}$
depends only on $\mu_{x}$ and $\mu_{y}$. As such, we can define the
additive free convolution operation $\boxplus$ on the space $\Sigma$
of all linear functionals on $\CC [X]$ via the formula $\mu_{x}
\boxplus \mu_{y} = \mu_{x+y}$. Finding $\mu_{x} \boxplus \mu_{y}$
given $\mu_{x}$ and $\mu_{y}$ is very nontrivial in practice and we
introduce some of the analytic tools that come to bear on this
problem.

Given a probability measure $\mu$ defined on $\RR$, we define the
Cauchy transform to be the analytic function, $G_{\mu} : \CC^{+}
\rightarrow \CC^{-}$ where $$G_{\mu} (z) = \int_{-\infty}^{\infty}
\frac{d\mu (t)}{z-t} .$$  For positive real numbers $\alpha$ and
$\beta$, let us set

$$\Gamma_{\alpha,\beta} := \{z\in \CC^{+}:\Im{(z)}>\alpha|\Re{(z)}|,
|z| > \beta.\}.$$

Note that for fixed $\alpha > 0$, there exists a $\beta >0$ so that
$G_{\mu} (z)$ maps $\Gamma_{\alpha, \beta}$ injectively onto a
region in $\mathbb{C}^{-}$ containing $z^{-1}$ for all $z \in
\Gamma_{\alpha ', \beta '}$ with $\alpha > \alpha '$ and $\beta >
\beta '$. We define the $R$-transform of $ \mu $ by the formula
$R_{\mu} (z) = G_{\mu}^{-1} (z) - z^{-1}$.  This function may be
defined for all $z \in \{z \in \mathbb{C}^{-}| -\Im{(z)} \geq \gamma
|\Re{(z)}| , 0 < \Im{(z)} < \lambda \}$ where $\gamma, \lambda \in
\mathbb{R}^{+}$ depend on our measure $\mu$, and maps into a set $\Gamma_{\alpha, \beta}$ as was previously defined. The $R$-transform
satisfies the property that $R_{\mu_{x} \boxplus \mu_{y}} (z) =
R_{\mu_{x}} (z) + R_{\mu_{y}} (z)$ on an appropriate angle, as above, where all are defined \cite{V2}.
A fundamental technique in finding $\mu_{x} \boxplus \mu_{y}$ when
given $\mu_{x}$ and $\mu_{y}$ for $x$ and $y$ freely independent
random variables is to compute the $R$-transform of each and then
find the probability measure whose $R$-transform corresponds to
their sum.

It turns out that it is advantageous to consider the reciprocal of
the Cauchy transform, $F_{\mu}(z) = G_{\mu}^{-1} (z)$.  We define
the function that corresponds to the $R$-transform in this setting,
$\varphi_{\mu} (z) := F_{\mu}^{-1} (z) - z$. We also define
$\mathbb{C}^{+}_{a} = \{z \in \mathbb{C}^{+} : \Im {(z)} > a \}$
where $a \in \mathbb{R}^{+}$.  The following properties were
established in \cite{Ma} and will be used throughout this paper.

\begin{lemma} \label{MAS1}
For any probability measure $\mu$ defined on $\RR$ having finite
variance, $|F_{\mu}(z) - z| \leq C / \Im (z)$ for $ z \in \CC ^{+}$
where $C>0$ depends only on $\mu$.
\end{lemma}

\begin{lemma}\label{MAS2}
For measure $\mu$ with zero mean and variance $\sigma ^{2}$,
$F_{\mu}^{-1} (z) : \CC_{2\sigma}^{+} \rightarrow \CC_{\sigma}^{+}$
is defined and satisfies $|\varphi_{\mu}(z)| = |F_{\mu}^{-1} (z) -
z| \leq 2\sigma ^{2} / \Im (z)$ for $z \in \CC_{2\sigma}^{+}$
\end{lemma}

\begin{lemma}\label{addlemma2}
For $\mu$ a probability measure on $\RR$ and $z \in \CC^{+}$, we
have that $\Im{F_{\mu}} (z) \geq \Im{z}$ with equality for some $z
\in \CC^{+}$ if and only if $\mu$ is a Dirac measure.
\end{lemma}

We also have the following property, first established for the
$R$-transform in \cite{V2} and which is easily seen to hold for
$\varphi_{\mu}$.

\begin{lemma}\label{addlemma}
For $\mu$ and $\nu$ probability measures on $\RR$, $\varphi_{\mu
\boxplus \nu} (z) = \varphi_{\mu} (z) + \varphi_{\nu} (z)$.
\end{lemma}

A measure is said to be $\boxplus$-infinitely divisible if, for all
$n \geq 2$, there exists a probably measure $\mu_{1/n}$ such that
$\mu = \mu_{1/n} ^{\boxplus ^{n}} $. These measures were introduced
in \cite{V2} and have proven to be very well behaved with respect to
free harmonic analysis.  The following properties, proved in \cite
{BV1}, exemplify this observation.

\begin{lemma} \label{infitely_divisible}

For $\mu$ an infinitely divisible measure on $\RR$ with variance
$\sigma^{2}$, the following properties hold:
\begin{enumerate}
\item
$F_{\mu} (z) + \varphi_{\mu} ( F_{\mu} (z)) = z$ for every $z\in \CC
^{+}$.
\item
The inequality  $|\varphi_{\mu}(z)|  \leq  2\sigma ^{2} / \Im {(z)}$
holds for all $z \in \CC ^{+}$.

\item
We have that
$$\varphi_{\mu}(z) = \alpha + \int_{-\infty}^{\infty} \frac{1 +
tz}{z-t} d\nu (t)$$ for $\nu$ is a probability measure.  Moreover,
if $\mu$ is not a dirac measure, we have that $\nu ([a,b])
> 0$ for some $a, b \in \RR$ and $\alpha \in \RR$.

\end{enumerate}

\end{lemma}

Recall the measure  $\gamma$ defined by the density function
$d\gamma (t) = (2 \pi)^{-1} \sqrt {4 - t^{2}} \chi _{(-2,2)} (t)
dt$. We refer to this measure as the $\textit{semicircle law}$.  The
semicircle law plays a role in free probability theory that is in
many ways analogous to that of the Gaussian law in classical
probability theory. Representative of this is the free central limit
theorem, first proved in \cite{V1}.

\begin{theorem}
Let $(A,\varphi)$ be a noncommutative probability space and
$\{a_{j}\}_{j=1}^{\infty} \subseteq A$ be a free family of random
variables such that:
\begin{enumerate}
\item
$\varphi (a_{j}) = 0$ for all $j \leq 1$
\item
$\sup_{j \leq 1} |\varphi (a_{j}^{k})| < \infty$ for each $k \leq 2$
\item
$\lim_{n \rightarrow \infty} n^{-1} \Sigma _{j=1}^{n} \varphi
(a_{j}^{2}) = 1$.
\end{enumerate}
Then $ n^{-1/2} (a_{1} + ... + a_{n})$ converges in distribution to
the semicircle law.
\end{theorem}

Consider noncommutative probability space $(A,\varphi)$ and assume
that $A$ is acting on a Hilbert space $H$. A self-adjoint operator
$T$ acting on $H$ is said to be affiliated with $A$ $($in symbols,
$T \eta A)$, if the spectral projections of $T$ belong to $A$.  We
note that our definition of distribution, $\mu_{T} (\sigma) :=
\varphi (E_{T}(\sigma))$ extends to these operators since, by
assumption, $E_{T}(\sigma) \in A$. These distributions will be
probability measures on $\RR$ whose support is not generally
bounded.  This is the class of measures we address in this paper.

\section{Uniform convergence on compact subsets of $(-2,2)$.}

Denote by $\mu_{n}$ the distribution of $n^{-1/2} (X_{1} + \cdots +
X_{n})$ where $X_{1} , \ldots , X_{n}$ are freely independent random
variables with distribution $\mu$.

The following lemmas are known and can be found, more or less
explicitly, in \cite{Ma}. Proofs are presented for the reader's
convenience.

\begin{lemma}\label{lemmagood}
For any $\alpha > 0$ there exists $\beta > 0$ so that
$\varphi_{\mu_{n}}$ is defined on $\Gamma_{\alpha,\beta}$ and
converges uniformly to $z^{-1}$ on compact subsets of
$\Gamma_{\alpha,\beta}$.
\end{lemma}
\begin{proof}
 As seen in \cite{Ma} $\varphi_{\mu_{n}}$ is defined on $\{z\in \CC^{+}:\Im{(z)} > 2\}$.
Thus, for fixed $\alpha > 0$, we need only pick $\beta$ big enough
so that $\Gamma_{\alpha,\beta}$ lies in the above set.

To prove the uniform convergence result, pick $\Omega \subseteq
\Gamma_{\alpha,\beta}$.  Note that $\{\varphi_{\mu_{n}}\}$ is a
normal family since Lemma \ref{MAS2} implies a uniform bound of $1$
on all of $\Gamma_{\alpha,\beta}$.  Therefore, we have subsequences
which converge uniformly on compact subsets of
$\Gamma_{\alpha,\beta}$ and we need only show that any such
subsequence converges to $z^{-1}$. Let $\varphi_{\mu_{n_{j}}}$
converge to $\varphi$. For $\gamma$, the semicircle measure,
$$|F_{\gamma}(z + \varphi(z)) - z| = |F_{\gamma}(z + \varphi(z)) - F_{\mu_{n_{j}}}(\varphi_{\mu_{n_{j}}}(z) +
z)|$$ $$\leq |F_{\gamma} (z + \varphi(z)) -
F_{\gamma}(\varphi_{\mu_{n_{j}}}(z) + z)| +
|F_{\gamma}(\varphi_{\mu_{n_{j}}}(z) + z) - F_{\mu_{n_{j}}}
(\varphi_{\mu_{n_{j}}}(z) + z)|$$ By the free central limit theorem,
$F_{\mu_{n_{j}}}$ converges to $F_{\gamma}$ uniformly on a
neighborhood of $z+ \varphi (z)$, and this implies that
$|F_{\gamma}(\varphi_{\mu_{n_{j}}}(z) + z) - F_{\mu_{n_{j}}}
(\varphi_{\mu_{n_{j}}}(z) + z)|$ converges to $0$.   Therefore, $z+
\varphi (z) = F_{\gamma}^{-1} (z)$ on $\Gamma_{\alpha,\beta}$ which
tells us that $\varphi = \varphi_{\gamma} = z^{-1}$.
\end{proof}

\begin{lemma}\label{newlemma}
For any ${\alpha, \beta} > 0$ and $n$ big enough, $\sqrt{n}
\varphi_{\mu} (\sqrt{n} z)$ is defined for all $z \in
\Gamma_{\alpha, \beta}$, agrees with $\varphi_{\mu_{n}}$ where both
are defined and converges to $z^{-1}$ uniformly on this set.
\end{lemma}
\begin{proof}
As was proven in \cite{Ma}, $\varphi_{\mu}$ is defined on $\{z\in
\CC^{+}:\Im{(z)} > 2\}$.  Pick $n$ large enough so that
$\Im{(\sqrt{n}z)} > 2$ for all $z \in \Gamma_{\alpha, \beta}$.  Then
$\sqrt{n} \varphi (\sqrt{n} z)$ is defined.

Now, by definition, $\varphi_{\mu_{n}} (F_{\mu_{n}} (z)) +
F_{\mu_{n}} (z) = z$ for all z with $\Im{(z)} > 2$.  Furthermore, it
was established in \cite{BV1} that $\varphi_{\mu_{n}} (z) = \sqrt{n}
\varphi_{\mu} (\sqrt{n} z)$ on this set.  Therefore, $\sqrt{n}
\varphi_{\mu} (\sqrt{n} F_{\mu_{n}} (z)) + F_{\mu_{n}} (z)$ is
defined on all of  $\Gamma_{\alpha, \beta}$ and is equal to the
identity for those $z$ with $\Im{(z)} > 2$.  By analytic
continuation, this implies that $\sqrt{n} \varphi_{\mu} (\sqrt{n}
F_{\mu_{n}} (z)) + F_{\mu_{n}} (z) = z$ for all $z \in
\Gamma_{\alpha, \beta}$.  This implies, by definition of
$\varphi_{\mu_{n}}$, that the two functions agree for those $z \in
\Gamma_{\alpha, \beta}$ where both are defined.

With regard to the question of uniform convergence on
$\Gamma_{\alpha, \beta}$, choose $\alpha$ , $\beta$, $\epsilon \ >
0$ and $\Omega \subseteq \Gamma_{\alpha, \beta}$ compact. By Lemma
\ref{lemmagood}, there exists a $\beta'$ $> 0$ so that
$\varphi_{\mu_{n}}$ is defined on $\Gamma_{\alpha,\beta'}$ for all
$n$ and converges to $z^{-1}$ uniformly on all compact subsets. Now,
pick $k$ such that $\sqrt{k} \Gamma_{\alpha, \beta} \subseteq
\Gamma_{\alpha,\beta'}$.  By the previous lemma, there exists $N >
0$ such that for all $n \geq N$  we have that
$$|\varphi_{\mu_{n}} (\alpha z) - (\alpha z)^{-1}|  <
(k+1)^{-1/2}\epsilon$$ for all $\alpha \in [k , k+1]$ and $z \in
\Omega$. We assume that $N
> k$ and, for $n \geq N$, consider numbers of the form $\alpha = \sqrt{k +
\ell n^{-1}}$ for $0 \leq \ell < k$.  We have that
$$|\varphi_{\mu_{n}} ((\sqrt{k
+ \ell n^{-1}}) z) - ((\sqrt{k + \ell n^{-1}}) z)^{-1} | <
(k+1)^{-1/2}\epsilon \leq (k + \ell n^{-1})^{-1/2} \epsilon$$
$$\Rightarrow  |\sqrt{k + \ell n^{-1}} \varphi_{\mu_{n}} ((\sqrt{k
+ \ell n^{-1}}) z) - z^{-1} | < \epsilon$$ and, since the previous
lemma implies that $\varphi_{\mu_{n}} (z) = \sqrt{n} \varphi_{\mu}
(\sqrt{n} z)$ for all $z \in \Gamma_{\alpha,\beta'}$, we have the
following:
$$|\sqrt{nk + \ell} \varphi_{\mu} ((\sqrt{nk + \ell}) z) - z^{-1} | < \epsilon $$ for all $z \in \Omega$ and $0 \leq
\ell < k$.  This implies that
$$|\sqrt{m} \varphi_{\mu} ((\sqrt{m}) z) - z^{-1} | < \epsilon $$ for
all $m \geq Nk$.  Thus, we have uniform convergence on compact
subsets of $\Gamma_{\alpha, \beta}$

For uniform convergence on all of $\Gamma_{\alpha,\beta}$, by Lemma
\ref{MAS2}, $\varphi_{\mu_{n}} (z)$ goes to zero uniformly as $|z|
\rightarrow \infty$ in $\Gamma_{\alpha,\beta}$.  As the same holds
for $z^{-1}$, by proving uniform convergence on compact sets, we
have the general result as well.
\end{proof}

\begin{theorem}\label{theoremA}
Let $\mu$ be a measure with mean $0$ and variance $1$.  Then
$d\mu_{n} / dt$ converges uniformly to the semicircle law on compact
subsets of $(-2,2)$.
\end{theorem}

\begin{proof}

It was shown in \cite{BV3} that $\mu_{n}$ is absolutely continuous
with respect to the Lebesgue measure for $n$ large enough.  We turn
our attention to the question of convergence.

Consider the interval $[-2 + \epsilon , 2 - \epsilon]$ and let
$\Lambda_{1} = \{ z : |z| = 1 , \Im{(z)} \geq \delta\}$ where
$\delta$ is chosen so that $\{ z + z^{-1} : z \in \Lambda_{1} \} =
[-2 + \epsilon , 2 - \epsilon]$.  Let $\Lambda_{2} = \{\lambda i : 1
< \lambda < 3 \}$. We denote by $\Omega$ a $2^{-1}\delta$
neighborhood of $\Lambda_{1} \cup \Lambda_{2}$.

Note that $\Omega \subset \Gamma_{\alpha , \beta}$ for appropriate
$\alpha$ and $\beta$.  Observe that $z + z^{-1}$ maps $\Omega$
conformally onto a neighborhood of $[-2 + \epsilon , 2 - \epsilon]
\cup i[0 , 2]$.  By lemma \ref{newlemma} and Rouche's theorem, the
same must hold for $\gamma_{n} (z) =
\sqrt{n}\varphi_{\mu}(\sqrt{n}z) + z$ when $n$ is large enough.
Thus, there exists a partition $\Omega = \Omega_{1} \cup \Omega_{2}
\cup \Omega_{3}$ with the following properties:

\begin{enumerate}

\item
$\Omega_{1}$ and $\Omega_{2}$ are open and connected.

\item
$\gamma_{n} (\Omega) \cap \CC^{+} = \gamma_{n} (\Omega_{1}) .$

\item
$\gamma_{n} (\Omega) \cap \CC^{-} = \gamma_{n} (\Omega_{2}) .$

\item
$\gamma_{n} (\Omega) \cap \RR = \gamma_{n} (\Omega_{3}) .$

\item
$2i \in \Omega_{1} .$

\end{enumerate}

Now, pick $t \in [-2 + \epsilon , 2 - \epsilon]$.  There exists a
positive real number $h_{t}$ and a path $z_{t}^{n} :[0 , h_{t}]
\rightarrow \Omega$ so that $\gamma_{n} (z_{t}^{n} (h) ) = t + ih$ .
Note that $z_{t}^{n} (h) \in \Omega_{1}$ for all $h \in [0 ,
h_{t}]$.   Lemmas \ref{lemmagood} and \ref{newlemma} imply that
$F_{\mu_{n}} (\gamma_{n} (z)) = z$ for those $z \in \Omega_{1}$ with
$\Im{(z)} > 2$.  As $\Omega_{1}$ is open and connected, this must
hold on the entire set by analytic continuation. Thus,
$$F_{\mu_{n}} (t + ih) = F_{\mu_{n}} (\gamma_{n} (z_{t}^{n} (h) )) =
z_{t}^{n} (h) .$$

Since $\gamma_{n}$ converges to $z + z^{-1}$ uniformly on $\Omega$,
we have that $z_{t}^{n} (0) \rightarrow (2^{-1}t  + i(\sqrt{1 -
(2^{-1}t )^{2}}) )$ uniformly over $t \in [-2 + \epsilon , 2 -
\epsilon]$ and $n \rightarrow \infty$.

Now, by the Stieltjes inversion formula,

$$\frac{d \mu_{n}}{dx} (t) = \lim_{h \downarrow 0} \frac{-1}{\pi} \Im{(G_{\mu_{n}} (t + ih
))} = \lim_{h \downarrow 0} \frac{-1}{\pi} \Im{(F_{\mu_{n}} (t + ih)
^{-1})}$$ $$ = \lim_{h \downarrow 0} \frac{-1}{\pi} \Im {(z_{n}^{t}
(h)^{-1})} = \frac{-1}{\pi} \Im {(z_{n}^{t} (0)^{-1})} \rightarrow
\frac{-1}{\pi}\Im{ (2^{-1}(t - \sqrt{t^{2} - 4}))} = \frac{1}{2\pi}
\sqrt{4 - t^{2}}
$$
As the convergence $z_{t}^{n} (0) \rightarrow ((2^{-1}t )^{2} +
i(\sqrt{1 - (2^{-1}t )^{2}}) )$ is uniform over $t \in [-2 +
\epsilon , 2 - \epsilon]$, we have that $\frac{d \mu_{n}}{dx} (t)
\rightarrow \frac{1}{2\pi} \sqrt{4 - t^{2}}$ uniformly over $t \in
[-2 + \epsilon , 2 - \epsilon]$, completing the proof.

\end{proof}

For any probability measure $\mu$, define the cumulative
distribution function $f_{\mu} (t) := \mu ((-\infty , t])$.  A
number of free analogues of Berry-Esseen have been proven with
respect to this function.  Most notably, in \cite{Ka}, it was shown
that for $\mu$ a measure with bounded support,
$$|f_{\mu_{n}} (t) - f_{\gamma} (t) | \leq CL^{3}n^{-1/2}$$ where C
is an absolute constant and supp$(\mu) \subseteq [-L,L]$.  It was
shown in \cite{CG} that $$|f_{\mu_{n}} (t) - f_{\gamma} (t)| \leq C
(|m_{3} (\mu)| + m_{4} (\mu)^{1/2}) n^{-1/2}$$ where $m_{k} (\mu)$
denotes the $kth$ moment of the measure $\mu$ and $C$ is and
absolute constant. From Theorem \ref{theoremA} we derive the
following partial result, stronger insofar as it makes no moment
assumptions beyond the second and does not require compact support
of the measure, but weaker since it does not provide a definite rate
of convergence with respect to $n$:

\begin{corollary}
For $\mu$ a probability measure with mean $0$ and variance $1$,
$|f_{\mu_{n}} (t) - f_{\gamma} (t)| \rightarrow 0$ uniformly as $ n
\rightarrow \infty$.
\end{corollary}

\begin{proof}
Choose $\epsilon > 0$.  Pick $\delta > 0$ such that $\gamma ([-2 +
\delta , 2 - \delta]) > 1- 4^{-1} \epsilon$.  By Theorem
\ref{theoremA}, there exists an $N$ such that for all $n \geq N$,
$|d\mu_{n} / dt (x) - d\gamma / dt (x)| < 8^{-1}\epsilon$ for all $x
\in [-2 + \delta , 2 - \delta]$.  For $\sigma \subseteq [-2 + \delta
, 2 - \delta]$, we have the following:
$$|\mu_{n} (\sigma) - \gamma (\sigma) |
\leq \int_{\sigma} |\frac{d_{\gamma}}{dt} (x) - {d_{\mu_{n}}} / {dt}
(x)| dt \leq  8^{-1} \epsilon |\sigma | \leq 2^{-1} \epsilon$$
Bearing in mind that $\mu_{n}$ is a probability measure, we have the
following:

\begin{enumerate}

\item
For $t \leq -2 + \delta$, we have that $f_{\mu_{n}}(t) = \mu_{n}
((-\infty , t])) = 1 - \mu_{n} ((t,\infty)) \leq 1 - \mu_{n} ([-2 +
\delta , 2 - \delta]) \leq 1 - (1 - \epsilon / 2) = \epsilon / 2$.

\item
For $t \in [-2 + \delta, 2 - \delta]$, we have that $f_{\mu_{n}}(t)
= \mu_{n} ((-\infty, -2 + \delta)) + \mu_{n} ([-2 + \delta, t]) \leq
\epsilon / 2 + \epsilon / 2$.

\item
For $t > 2 - \delta$, we have that $f_{\mu_{n}} (t) \geq \mu_{n}
([-2 + \delta, 2 - \delta]) > 1 - \epsilon$.

\end{enumerate}
Thus, our claim holds.

\end{proof}

It would be interesting to see whether $|f_{\mu_{n}} (t) -
f_{\gamma} (t)| \leq C n^{-1/2}$, for $C$ an absolute constant, with
no assumptions on the boundedness of the support and no assumptions
on the existence of moments beyond the second.

\section{Uniform convergence for infinitely divisible measures.}

We begin with a lemma before proving the main result of the section.
\begin{lemma} \label{lastlemma}
Let $\mu$ be an infinitely divisible measure with mean $0$ and
variance $1$.  Then there exists a $C > 0$ so that for n big enough,
$$F_{\mu_{n}} (\CC^{+}) \subseteq \{z \in \CC^{+} : |z|
> C \}$$
\end{lemma}

\begin{proof}
By Lemma \ref{infitely_divisible} $(2)$, we see that
$$\Im (\varphi_{\mu_{n}}(z)) = \Im (\sqrt{n} \varphi_{\mu}
(\sqrt{n}z)) = \sqrt{n} \Im{\int_{-\infty}^{\infty} \frac{1 +
\sqrt{n}tz}{\sqrt{n}z-t} d\nu (t)} = -\int_{-\infty}^{\infty}
\frac{y(1 + t^{2})}{(x - n^{-1/2}t)^{2} + y^{2}} d\nu(t)$$ where $z
= x + iy$.  Now for $\nu ([a,b]) > 0$, we have the following:
$$-\Im (\varphi_{\mu_{n}}(z)) = \int_{-\infty}^{\infty} \frac{y(1 + t^{2})}{(x -
n^{-1/2}t)^{2} + y^{2}} d\nu(t) \geq \int_{a}^{b} \frac{y(1 +
t^{2})}{(x - n^{-1/2}t)^{2} + y^{2}} d\nu(t) $$ $$\geq \frac{y
(\nu([a,b]))}{\max\{|z-n^{-1/2}|b||^{2}, |z-n^{-1/2}|a||^{2}\}}$$

Now, for $z\in \CC^{+}$, by Lemma \ref{infitely_divisible} $(1)$,
$F_{\mu_{n}}(z) + \varphi_{\mu_{n}}(F_{\mu_{n}}(z)) = z$ which
implies that $\Im(\varphi_{\mu_{n}}(F_{\mu_{n}}(z))) +
\Im(F_{\mu_{n}}(z)) > 0$. Therefore,
$-\Im(\varphi_{\mu_{n}}(F_{\mu_{n}}(z))) < \Im(F_{\mu_{n}}(z))$ and,
setting $F_{\mu_{n}} (z) = x' + iy'$, we have that:
 $$\frac{y'\nu([a,b])}{\max\{|F_{\mu_{n}}(z) - |z-n^{-1/2}|b||^{2},
|F_{\mu_{n}}(z) - |z-n^{-1/2}|b||^{2}\}} \leq y'$$ $\Rightarrow$
$\nu([a,b]) \leq \max\{|F_{\mu_{n}}(z)- n^{-1/2} |b| \ |^{2},
|F_{\mu_{n}}(z)- n^{-1/2} |a| \ |^{2}\}$.  As $n^{-1/2}|a| \ ,
n^{-1/2}|b| \rightarrow 0$, we get that for n big enough, $C =
\sqrt{\nu([a,b]) / 2} \leq |F_{\mu_{n}}(z)|$ , which proves our
lemma.

\end{proof}

We now formulate and prove our main theorem.

\begin{theorem} \label{maintheorem}
Let $\mu$ be an infinitely divisible measure with mean $0$ and
variance $1$.  Then $\frac{d\mu_{n}}{dt}$ converges to the
semicircle law uniformly.
\end{theorem}

\begin{proof}
We already know from Theorem \ref{theoremA} that $\lim_{n
\rightarrow \infty} d\mu_{n} (x) = (2\pi)^{-1} \sqrt {4 -x^{2}}$
uniformly on compact subintervals of $(-2,2)$.  Assuming, for the
sake of contradiction, that we do not have uniform convergence of
the density, we get a sequence of real numbers, $t_{k}$ with the
following properties:

\begin{enumerate}
\item
$\liminf_{k \rightarrow \infty} |t_{k}| \geq 2$
\item
There exists an $\eta > 0$ so that $\frac{d\mu_{n_{k}}}{dt} (t_{k})
> \eta$ for a sequence of natural numbers, $n_{k} \uparrow \infty$.

\end{enumerate}

By Stieltjes inversion formula, $\exists \ h_{k} > 0 \ \ s.t. \
\forall \ h \in (0,h_{k})$ we have the following:
$$\frac{\Im{({F_{\mu_{n_{k}}}(t_{n_{k}} + ih)})}}{|F_{\mu_{n_{k}}}(t_{n_{k}} + ih)|^{2}} = -\Im{(G_{\mu_{n_{k}}}(t_{n_{k}} + ih))} > \pi \eta$$
which, coupled with Lemma \ref{lastlemma}, implies
\begin{equation}\label{I}
 \eta \pi C^{2}
\leq \eta \pi |F_{\mu_{n_{k}}} (t_{k} + ih) |^{2} <
\Im{({F_{\mu_{n_{k}}}(t_{k} + ih)})}.
\end{equation}

Recall that Lemma \ref{infitely_divisible} implies that
$F_{\mu_{n_{k}}}(t_{k} + ih) +
\varphi_{\mu_{n_{k}}}(F_{\mu_{n_{k}}}(t_{k} + ih)) = t_{k} + ih$.
This tells us that  $\ \Im{(F_{\mu_{n_{k}}}(t_{k} + ih))} +
\Im{(\varphi_{\mu_{n_{k}}}(F_{\mu_{n_{k}}}(t_{k} + ih)))} = h$. Now,
with $a \in \RR$ as in Lemma \ref{infitely_divisible} $(3)$, we have
that
$$\Im{(\varphi_{\mu_{n_{k}}}(F_{\mu_{n_{k}}}(t_{k} + ih)))} \leq
| \varphi_{\mu_{n_{k}}}(F_{\mu_{n_{k}}}(t_{k} + ih)) - a| \leq
\frac{2}{\Im{(F_{\mu_{n_{k}}}(t_{k} + ih))}}.$$  Furthermore, Lemma
\ref{addlemma2} tells us that $\Im{(F_{\mu_{n_{k}})}(t_{k} + ih)} >
0$ and we have the following:

$$ h  = \Im{(F_{\mu_{n_{k}}}(t_{k} + ih))} +
\Im{(\varphi_{\mu_{n_{k}}}(F_{\mu_{n_{k}}}(t_{k} + ih)))}  \geq
\Im{(F_{\mu_{n_{k}}}(t_{k} + ih))} - |
\varphi_{\mu_{n_{k}}}(F_{\mu_{n_{k}}}(t_{k} + ih))| $$ $$ \geq
\Im{(F_{\mu_{n_{k}}}(t_{k} + ih))} -
\frac{2}{\Im{(F_{\mu_{n_{k}}}(t_{k} + ih))}}.$$ This implies that
\begin{equation}\label{II}
\Im{(F_{\mu_{n_{k}}}(t_{k} + ih))} \leq M
\end{equation}
where M is a constant and $h$ is sufficiently small, independent of
$t_{k}$.

Thus, by (\ref{I}) and (\ref{II}) , ${(F_{\mu_{n_{k}}}(t_{k} +
ih))}$ lies entirely in the truncated disk, $\Omega = \{z: | z | <
K,  \ \Im{(z)} > \eta \pi C^{2} \}$.  By infinite divisibility,
$\varphi_{\mu_{n_{k}}}$ is defined on $F_{\mu_{n_{k}}} (\CC^{+})
\cap \Omega$ and, by Lemma \ref{newlemma}, converges to $z^{-1}$
uniformly on this set.  However, $\{ z + z^{-1} : z\in \Omega \}$
contains no neighborhood of $\RR$ outside of $[-2 + \delta , 2 -
\delta]$ for some fixed $\delta > 0$.  The same must also hold for
$z + \varphi_{\mu_{n_{k}}} (z)$ for $n_{k}$ big enough. Therefore,
$t_{k} + ih = F_{\mu_{n_{k}}}(t_{k} + ih) +
\varphi_{\mu_{n_{k}}}(F_{\mu_{n_{k}}}(t_{k} + ih))$ must be
contained in a neighborhood of $[-2 + \delta, 2 - \delta]$ in $\CC$
of as small a size as we wish for h small enough and $k$ big enough.
But this contradicts the fact the $t_{k}$ must eventually leave $[-2
+ \frac{\delta}{2}, 2 - \frac{\delta}{2}]$. Thus, our theorem holds.

\end{proof}

It would be interesting to see whether Lemma $\ref{lastlemma}$ or
Theorem \ref{maintheorem} holds without the assumption that $\mu$ is
infinitely divisible.

\bibliographystyle{abbrv}
\bibliography{firstpaperref}

\end{document}